\documentclass[portuges,12pt,letter]{article}
\usepackage[centertags]{amsmath}
\usepackage{amsfonts}
\usepackage{newlfont}
\usepackage{amscd}
\usepackage{graphics}
\usepackage{epsfig}
\usepackage{indentfirst}
\usepackage{amsxtra}
\usepackage[latin1]{inputenc}
\usepackage{amssymb, amsmath}
\usepackage{amsthm}
\usepackage[mathscr]{eucal}

\newtheorem{thm}{Theorem}[section]

\newtheorem{obe}[thm]{Remark}

\author{Fabio Silva Botelho \\ Department of Mathematics \\  Federal University of Santa Catarina, UFSC \\
Florian\'{o}polis, SC - Brazil}

\title{\bf  A primal dual formulation through a proximal approach for non-convex variational optimization}
\date{}
\begin{document}
\maketitle

\abstract{This article develops a primal dual formulation for a primal proximal approach suitable for a large class of non-convex models in the calculus of variations. The results are established through standard tools of functional analysis, convex analysis and duality theory and are applied to a Ginzburg-Landau type model. Finally, in the last two sections, we present concerning optimality conditions and another related duality principle for the model in question.}

\section{Introduction}
We start this article by justifying the suitability of the proximal approach for the concerning model.

Consider a domain $\Omega \subset \mathbb{R}^3$ and the functional $J:U \rightarrow \mathbb{R}$ where
\begin{eqnarray}J(u)&=& \frac{\gamma}{2}\int_\Omega \nabla u \cdot \nabla u\;dx +\frac{\alpha}{2}\int_\Omega (u^2-\beta)^2\;dx
\nonumber \\ && -\langle u,f \rangle_{L^2},\; \forall u \in U=W_0^{1,2}(\Omega).
\end{eqnarray}

We could write such a functional as
$$J(u)=G_1(u,0)+F_1(u),\; \forall u \in U,$$
where
$$G_1(u,v)=\frac{\gamma}{2}\int_\Omega \nabla u \cdot \nabla u\;dx +\frac{\alpha}{2}\int_\Omega (u^2-\beta+v)^2\;dx
-\frac{\varepsilon}{2}\int_\Omega u^2\;dx,$$
and
$$F_1(u)=\frac{\varepsilon}{2}\int_\Omega u^2\;dx-\langle u,f \rangle_{L^2}.$$

Among other possibilities, we could define the dual functional as

$$J^*(v^*,v_0^*)=-G^*_1(v^*,v_0^*)-F^*(v^*),$$

where $$G_1^*(v^*,v_0^*)= \frac{1}{2} \int_\Omega \frac{(v^*)^2}{(-\gamma \nabla^2+2v_0^*-\varepsilon)}\;dx
+\frac{1}{2\alpha}\int_\Omega (v_0^*)^2\;dx+\beta \int_\Omega v_0^*\;dx,$$
and
$$F^*_1(v^*)=\frac{1}{2 \varepsilon}\int_\Omega (v^*-f)^2\;dx$$

Through the variation in $v_0^*$ we obtain

$$\frac{(v^*)^2}{(-\gamma \nabla^2 +2 v_0^*-\varepsilon)^2}-\frac{v_0^*}{\alpha}-\beta=\mathbf{0},$$
intending to obtain conditions for a solution $v_0^*(v^*)$ and thus to obtain a final functional as a function of $v^*$ with a possible large region of convexity  (in fact concavity) due the term
$$F^*_1(v^*)=\frac{1}{2 \varepsilon}\int_\Omega (v^*-f)^2\;dx$$ with a small value for $\varepsilon>0$.

The issue is that if the term $$-\gamma \nabla^2+2v_0^* -\varepsilon$$ corresponds to an undefined matrix (this is a common situation for the case of local minima for the primal formulation) we may not have the hypothesis of the implicit function theorem satisfied so that critical points of the dual formulation may not correspond to critical points of the primal one and reciprocally.

Indeed, we may obtain for the second variation of $J^*$ in $v_0^*$
$$\frac{\partial^2 J^*(v_0^*)}{\partial (v_0^*)^2}=-4\frac{(v^*)^2}{(-\gamma \nabla^2+2v_0^*-\varepsilon)^3}-\frac{1}{\alpha},$$

Observe that for a critical point denoting $$u=\frac{(v^*)}{(-\gamma \nabla^2+2v_0^*-\varepsilon)}$$ we have
$$v_0^*=\alpha\left( \left(\frac{(v^*)}{(-\gamma \nabla^2+2v_0^*-\varepsilon)}\right)^2-\beta\right)=\alpha(u^2-\beta),$$
so that
$$\frac{\partial^2 J^*(v_0^*)}{\partial (v_0^*)^2}=-4\frac{u^2}{(-\gamma \nabla^2+2v_0^*-\varepsilon)}-\frac{1}{\alpha}$$
and thus
$$\frac{\partial^2 J^*(v_0^*)}{\partial (v_0^*)^2}=\frac{-4\alpha u^2+\gamma \nabla^2-2v_0^*+\varepsilon}{(-\gamma \nabla^2+2v_0^*-\varepsilon) \alpha} =\frac{-\delta^2J(u)+\varepsilon}{(-\gamma \nabla^2+2v_0^*-\varepsilon) \alpha}.$$

Therefore if for a critical point where $$\delta^2 J(u)-\varepsilon > \mathbf{0}$$the term
$$-\gamma \nabla^2+2v_0^*-\varepsilon$$ corresponds to an undefined matrix, we have that $$\frac{\partial^2 J^*(v_0^*)}{\partial (v_0^*)^2}$$ is also undefined and the hypothesis of the implicit function theorem may not be satisfied, in order to obtain $v_0^*(v^*)$. The other issue is that $$\delta^2 J(v^*,v_0^*)$$ may  also be undefined at a critical point, so that we do not have a qualitative correspondence between the primal and dual critical points.

So this may lead us, for a large class of similar models, through such a formulation, to wrong results concerning the equivalence of critical points for the primal and dual formulations.

In order to solve this problem, in this article we propose a kind of proximal variational formulation with exact penalization. Thus, with such facts in mind, we propose as the primal dual equivalent formulation for the original primal problem in question, the following functional $\hat{J}:U \times Y \rightarrow \mathbb{R},$ where
\begin{eqnarray}
\hat{J}(u,p)&=&\frac{\gamma}{2}\int_\Omega \nabla u \cdot \nabla u\;dx +\frac{\alpha}{2}\int_\Omega (u^2-\beta)^2\;dx
\nonumber \\ && +\frac{K}{2}\int_{\Omega}(u-p)^2\;dx-\langle u,f \rangle_{L^2}
\end{eqnarray}

We highlight the proximal term $$\frac{K}{2}\int_{\Omega}(u-p)^2\;dx$$ makes the primal formulation  convex in $u$ for appropriate values of $K>0$.

In the next section  we present the theoretical results for a duality principle concerning such a proximal formulation. We believe through an analysis of the proof of the next theorem the suitability of such a proximal
formulation will be clarified.

\begin{obe}  About the references, in our work we have been greatly influenced by the works of J.J. Telega and W.R. Bielski, in particular by \cite{85,2900}. The duality principle here developed  for the proximal approach is also inspired by the works  J.F. Toland \cite{12} and Ekeland and Temam \cite{[6]}.

Related problems are addressed in \cite{700, 120,360}. About the physics of the problem in question we would cite
\cite{100} and \cite{101}. Details on the Sobolev spaces involved may be found in \cite{1,700}.
\end{obe}

\begin{obe} Even though we have not relabeled the functionals and operators, we shall consider a finite dimensional approximation for the model in question, in a finite elements or finite differences context.

In such a finite elements or finite differences context, we emphasize that the  notation
$$\int_\Omega \frac{(v^*_1)^2}{-\gamma \nabla^2+K+\varepsilon}\;dx$$  stands for
$$\left\langle \left(-\gamma \nabla^2+K I_d+\varepsilon I_d\right)^{-1}v_1^*,v_1^* \right\rangle$$ where $I_d$ denotes the identity matrix in an appropriate finite dimensional approximate space.
\end{obe}
\begin{obe} Finally we highlight that for invertible $n \times n$ matrices or invertible linear operators $A$ and $B$  we have
$$-A^{-1}+B^{-1}= A^{-1}(B-A)B^{-1}$$ and sometimes, as the meaning is clear, we may simply denote
$$-A^{-1}+B^{-1}=\frac{B-A}{AB}.$$
\end{obe}
\section{The main duality principle}
In this section we present the main result in this article, which is summarized by the next theorem.

At this point we highlight that the optimality criterion presented in the item \ref{a3b3} in  the next theorem, namely $$-\gamma\nabla^2+2v_0^*>0,$$ may be found in  analogous form in the article \cite{15} which was published in 2010, but in fact submitted in August of the year 2007, as indicated in the concerning Journal web-site.  Related results on duality theory may be originally found in \cite{360}.

\begin{thm}\label{a11b11} Let $\Omega \subset \mathbb{R}^3$ be an open, bounded and connected set with a regular (Lipschitzian)
boundary denoted by $\partial \Omega$. Even though we have not relabeled the functionals and operators,  consider a finite dimensional approximation for the model in question, in a finite elements or finite differences context, where we define the functionals $\hat{J}: U \times Y \rightarrow \mathbb{R}$
and $J:U \rightarrow \mathbb{R}$,  by
\begin{eqnarray}
\hat{J}(u,p)&=&\frac{\gamma}{2}\int_\Omega \nabla u \cdot \nabla u\;dx +\frac{\alpha}{2}\int_\Omega (u^2-\beta)^2\;dx
\nonumber \\ && +\frac{K}{2}\int_{\Omega}(u-p)^2\;dx-\langle u,f \rangle_{L^2}
\end{eqnarray}
and $$J(u)=\hat{J}(u,u),$$
where
$$U=W_0^{1,2}(\Omega),$$
$$Y=Y^*=L^2(\Omega),$$
$\alpha>0,\beta>0,\gamma>0$, $K>0$ and $f \in C^1(\overline{\Omega}).$

Furthermore, for a sufficiently small parameter $\varepsilon>0$, define $G:U \times Y \times Y \rightarrow \mathbb{R}$ by
\begin{eqnarray}
G(u,v,p)&=&\frac{\gamma}{2}\int_\Omega \nabla u \cdot \nabla u\;dx +\frac{\alpha}{2}\int_\Omega (u^2-\beta+v)^2\;dx
\nonumber \\ && -\langle u,K p\rangle_{L^2} +\frac{K}{2}\int_\Omega u^2\;dx+\frac{\varepsilon}{2}\int_\Omega u^2\;dx,
\end{eqnarray}
 $F:U \rightarrow \mathbb{R}$ by
$$F(u)=\frac{\varepsilon}{2}\int_\Omega u^2\;dx+\langle u,f \rangle_{L^2}$$
and $H:Y \rightarrow \mathbb{R}$ by
$$H(p)=\frac{K}{2}\int_\Omega p^2\;dx,$$
so that
$$\hat{J}(u,p)=G(u,0,p)-F(u)+H(p),\; \forall (u,p) \in U \times Y.$$

Define also, $G^*:Y^*\times Y^* \times Y\rightarrow \mathbb{R}$ by
\begin{eqnarray}G^*(v^*,v_0^*,p)&=&\sup_{u \in U}\sup_{v \in Y}\{ \langle u,v^*\rangle_{L^2}+\langle v,v_0^*\rangle_{L^2}-G(u,v,p)\}
\nonumber \\ &=& \frac{1}{2}\int_\Omega \frac{(v^*+K p)^2}{-\gamma \nabla^2+2 v_0^*+K+\varepsilon}\;dx
\nonumber \\ && +\frac{1}{2 \alpha}\int_\Omega (v_0^*)^2\;dx+\beta \int_\Omega v_0^*\;dx,\end{eqnarray}
if $v_0^* \in B^*$ where
$$B^*=\left\{v_0^* \in Y^* \;:\;-\gamma \nabla^2+ 2v_0^*+K+\varepsilon > \frac{K}{2}\right\},$$
$F^*:Y^* \rightarrow \mathbb{R}$ where
\begin{eqnarray}F^*(v^*)&=&\sup_{ u \in Y}\{\langle u, v^* \rangle_{L^2}-F(u)\} \nonumber \\ &=&
\frac{1}{2 \varepsilon}\int_\Omega (v^*-f)^2\;dx.\end{eqnarray}
and $J^*: Y^*\times B^* \times Y \rightarrow \mathbb{R}$ by
$$J^*(v^*,v_0^*,p)=-G^*(v^*,v_0^*,p)+F^*(v^*)+H(p), \; \forall (v^*,v_0^*, p) \in Y^* \times B^* \times Y.$$

Under such hypotheses,

\begin{enumerate}
\item Assume $u_0 \in U$ is such that $\delta J(u_0)=\mathbf{0}$ and define
$$\hat{v}_0^*=\alpha (u_0^2-\beta),$$
$$\hat{v}^*=\varepsilon u_0+f,$$
$$\hat{p}=u_0$$

under such assumptions, $$\delta J^*(\hat{v}^*,\hat{v}_0^*, \hat{p})= \mathbf{0}.$$

\begin{enumerate}
\item\label{a1b1} Assume also $\delta^2J(u_0)> \mathbf{0}$ and $\hat{v}_0^* \in B^*$.
Under such additional hypotheses, there exist $r_1,r_2,r_3>0$ such that
\begin{eqnarray}
J(u_0)&=& \inf_{u \in B_{r_1}(u_0)} J(u) \nonumber \\ &=&
\inf_{v^* \in B_{r_3}(\hat{v}^*)}\left\{\inf_{p \in B_{r_2}(\hat{p})}\left\{\sup_{v_0^* \in B^*}J^*(v^*,v_0^*,p)\right\}\right\}
\nonumber \\ &=& J^*(\hat{v}^*,\hat{v}_0^*, \hat{p}).
\end{eqnarray}

Moreover, defining $J_3^*:B_{r_3}(\hat{v}^*) \rightarrow \mathbb{R}$ by
$$J_3^*(v^*)=\inf_{p \in B_{r_2}(\hat{p})}\left\{\sup_{v_0^* \in B^*}J^*(v^*,v_0^*,p)\right\}$$
we have that
$$\delta J^*_3(\hat{v}^*)=\mathbf{0}$$
$$\delta^2 J_3^*(\hat{v}^*) > \mathbf{0}$$ so that
\begin{eqnarray}
J(u_0)&=& \inf_{u \in B_{r_1}(u_0)}J(u) \nonumber \\ &=& \inf_{v^* \in B_{r_3}(\hat{v}^*)} J_3^*(v^*) \nonumber \\
&=& J_3^*(\hat{v}^*).
\end{eqnarray}
\item\label{a2b2} Suppose $\delta^2J(u_0)< \mathbf{0}$ and $\hat{v}_0^* \in B^*$.
Under such additional hypotheses, there exist $r_1,r_2,r_3>0$ such that
\begin{eqnarray}
J(u_0)&=& \sup_{u \in B_{r_1}(u_0)} J(u) \nonumber \\ &=&
\inf_{v^* \in B_{r_3}(\hat{v}^*)}\left\{\sup_{p \in B_{r_2}(\hat{p})}\left\{\sup_{v_0^* \in B^*}J^*(v^*,v_0^*,p)
\right\} \right\}
\nonumber \\ &=& J^*(\hat{v}^*,\hat{v}_0^*, \hat{p}).
\end{eqnarray}

Moreover, defining $J_5^*:B_{r_3}(\hat{v}^*) \rightarrow \mathbb{R}$ by
$$J_5^*(v^*)=\sup_{p \in B_{r_2}(\hat{p})}\left\{\sup_{v_0^* \in B^*}J^*(v^*,v_0^*,p)\right\}$$
we have
$$\delta J_5^*(\hat{v}^*)=\mathbf{0}$$
$$\delta^2 J_5^*(\hat{v}^*) >\mathbf{0}$$ so that
\begin{eqnarray}
J(u_0)&=& \sup_{u \in B_{r_1}(u_0)}J(u) \nonumber \\ &=& \inf_{v^* \in B_{r_3}(\hat{v}^*)} J_5^*(v^*) \nonumber \\
&=& J_5^*(\hat{v}^*).
\end{eqnarray}
\item\label{a3b3} For this item define $A^+$ by
$$A^+=\{v_0^* \in Y^*\;:\; -\gamma \nabla^2+2v_0^* >\mathbf{0}\}.$$
Assume $\hat{v}_0^* \in A^+ \cap B^*$.

Under such additional assumptions and definitions, we have
\begin{eqnarray}
J(u_0)&=& \inf_{u \in U} J(u) \nonumber \\ &=&
\inf_{(v^*,p) \in Y^* \times Y}\left\{\sup_{v_0^* \in A^+ \cap B^*}J^*(v^*,v_0^*,p)\right\}
\nonumber \\ &=& J^*(\hat{v}^*,\hat{v}_0^*, \hat{p}).
\end{eqnarray}

Moreover, defining $J_7^*:Y^*\times Y \rightarrow \mathbb{R}$ by
$$J_7^*(v^*,p)=\left\{\sup_{v_0^* \in A^+ \cap B^*}J^*(v^*,v_0^*,p)\right\}$$
we have
$$\delta J_7^*(\hat{v}^*,\hat{p})=\mathbf{0}$$
$$\delta^2 J_7^*(\hat{v}^*,\hat{p}) > \mathbf{0}$$ so that
\begin{eqnarray}
J(u_0)&=& \inf_{u \in U}J(u) \nonumber \\ &=& \inf_{(v^*,p) \in Y^* \times Y} J_7^*(v^*,p) \nonumber \\
&=& J_7^*(\hat{v}^*,\hat{p}).
\end{eqnarray}

\end{enumerate}
\end{enumerate}

\end{thm}
\begin{proof}
Suppose $u_0 \in U$ is such that $\delta J(u_0)=\mathbf{0}.$

We shall start by proving that $$\delta J^*(\hat{v}^*, \hat{v}_0^*, \hat{p})=\mathbf{0}.$$

Observe that from
$$\delta J(u_0)=0$$ we have that
$$-\gamma \nabla^2 u_0+2\alpha (u_0^2-\beta)u_0-f=0, \text{ in } \Omega,$$ so that
$$-\gamma \nabla^2 u_0+2\alpha (u_0^2-\beta)u_0-\varepsilon u_0+K u_0+\varepsilon u_0-K u_0-f=0,$$
that is
\begin{eqnarray}
\hat{v}^*+K \hat{p}=\varepsilon u_0+f+K u_0= -\gamma \nabla^2 u_0+2\alpha (u_0^2-\beta)u_0+\varepsilon u_0+K u_0.
\end{eqnarray}
Thus,
$$u_0 = \frac{\hat{v}^*+K \hat{p}}{-\gamma \nabla^2+2 \hat{v}_0^*+K+\varepsilon},$$ so that
$$u_0= \frac{\hat{v}^*-f}{\varepsilon}=\frac{\hat{v}^*+K \hat{p}}{-\gamma \nabla^2+2 \hat{v}_0^*+K+\varepsilon}.$$

Therefore $$ \frac{\hat{v}^*-f}{\varepsilon}-\frac{\hat{v}^*+K \hat{p}}{-\gamma \nabla^2+2 \hat{v}_0^*+K+\varepsilon}=\mathbf{0},$$ and consequently we may infer that
$$\frac{\partial  J^*(\hat{v}^*,\hat{v}_0^*, \hat{p})}{\partial v^*}=\mathbf{0}.$$
On the other hand
$$\frac{\hat{v}_0^*}{\alpha}=(u_0^2-\beta)=\left(\frac{\hat{v}^*+K \hat{p}}{-\gamma \nabla^2+2 \hat{v}_0^*+K+\varepsilon}\right)^2-\beta,$$ so that
$$-\frac{\hat{v}_0^*}{\alpha}+\left(\frac{\hat{v}^*+K \hat{p}}{-\gamma \nabla^2+2 \hat{v}_0^*+K+\varepsilon}\right)^2-\beta=0,$$  that is,
$$\frac{\partial  J^*(\hat{v}^*,\hat{v}_0^*, \hat{p})}{\partial \hat{v}^*_0}=\mathbf{0}.$$

Moreover
$$K \hat{p}=K u_0=K\left(\frac{\hat{v}^*+K \hat{p}}{-\gamma \nabla^2+2 \hat{v}_0^*+K+\varepsilon}\right),$$
so that
$$K \hat{p}-K\left(\frac{\hat{v}^*+K \hat{p}}{-\gamma \nabla^2+2 \hat{v}_0^*+K+\varepsilon}\right)=0,$$
that is,
$$\frac{\partial  J^*(\hat{v}^*,\hat{v}_0^*, \hat{p})}{\partial p}=\mathbf{0}.$$

From these last results, we have that
$$\delta J^*(\hat{v}^*,\hat{v}_0^*,\hat{p})=\mathbf{0}.$$

Also
\begin{eqnarray}
\frac{\partial J^*_3(\hat{v}^*)}{\partial v^*}&=&\frac{\partial  J^*(\hat{v}^*,\hat{v}_0^*, \hat{p})}{\partial v^*}
\nonumber \\ &&+\frac{\partial  J^*(\hat{v}^*,\hat{v}_0^*, \hat{p})}{\partial v^*_0}\frac{\partial \hat{v}^*_0}{\partial v^*} \nonumber \\ &&+ \frac{\partial  J^*(\hat{v}^*,\hat{v}_0^*, \hat{p})}{\partial p}\frac{\partial \hat{p}}{\partial v^*} \nonumber \\ &=& \mathbf{0}.
\end{eqnarray}

Similarly we may obtain
$$\frac{\partial J^*_5(\hat{v}^*)}{\partial v^*}=\mathbf{0},$$ and
$$\delta J^*_7(\hat{v}^*,\hat{p})=\mathbf{0}.$$

From the relations between the primal and dual variables, as a by-product of the Legendre transform proprieties we may obtain

\begin{eqnarray}
&&J^*(\hat{v}^*,\hat{v}_0^*,\hat{p}) \nonumber \\ &=&
-G^*(\hat{v}^*,\hat{v}_0^*,\hat{p})+F^*(\hat{v}^*)+H(\hat{p}) \nonumber \\ &=&
G(u_0,\mathbf{0},\hat{p})-F(u_0)+H(\hat{p})
\nonumber \\ &=& \hat{J}(u_0,\hat{p})\nonumber \\ &=& J(u_0).
\end{eqnarray}

Suppose now $$\delta^2 J(u_0)>\mathbf{0}.$$

Define $J_8^*: Y^*\times Y \rightarrow \mathbb{R}$ by
$$J_8^*(v^*,p)=\sup_{v_0^* \in B^*} J^*(v^*,v_0^*,p).$$

In particular we have got
$$J_8^*(\hat{v}^*,\hat{p})=\sup_{v_0^* \in B^*} J^*(\hat{v}^*,v_0^*,\hat{p})=J^*(\hat{v}^*,\hat{v}_0^*.\hat{p}).$$

Observe that
\begin{eqnarray}
\frac{\partial^2 J_8^*(\hat{v}^*, \hat{p})}{\partial p^2}&=&
\frac{\partial^2 J^*(\hat{v}^*,\hat{v}_0^*,\hat{p})}{\partial p^2} \nonumber \\ &&+
\frac{\partial^2 J^*(\hat{v}^*,\hat{v}_0^*,\hat{p})}{\partial p \partial v_0^*}\frac{\partial \hat{v}_0^*}{\partial p}.
\end{eqnarray}

At this point we recall that
$$\frac{\partial J^*(\hat{v}^*,\hat{v}_0^*,\hat{p})}{ \partial v_0^*}=\mathbf{0},$$ so that
$$\left(\frac{\hat{v}^*+K \hat{p}}{-\gamma \nabla^2+2 \hat{v}_0^*+K+\varepsilon}\right)^2-\frac{\hat{v}_0^*}{\alpha}
-\beta=\mathbf{0}$$

Hence, taking the variation in $p$ of such a last equation, we obtain
\begin{eqnarray}
&&\frac{2K(\hat{v}^*+K \hat{p})}{(-\gamma \nabla^2+2 \hat{v}_0^*+K+\varepsilon)^2}
\nonumber \\ &&-4\frac{(\hat{v}^*+K \hat{p})^2}{(-\gamma \nabla^2+2 \hat{v}_0^*+K+\varepsilon)^3}\frac{\partial \hat{v}_0^*}{\partial p} \nonumber \\ &&-\frac{1}{\alpha}\frac{\partial \hat{v}_0^*}{\partial p}=\mathbf{0}.
\end{eqnarray}
so that
\begin{eqnarray}
&&\frac{2Ku_0}{(-\gamma \nabla^2+2 \hat{v}_0^*+K+\varepsilon)}
\nonumber \\ &&-4\frac{(u_0)^2}{(-\gamma \nabla^2+2 \hat{v}_0^*+K+\varepsilon)}\frac{\partial \hat{v}_0^*}{\partial p} \nonumber \\ &&-\frac{1}{\alpha}\frac{\partial \hat{v}_0^*}{\partial p}=\mathbf{0}.
\end{eqnarray}
and thus
$$\frac{\partial \hat{v}_0^*}{\partial p}=\frac{2\alpha K u_0}{(-\gamma \nabla^2+4 \alpha u_0^2+2 \hat{v}_0^*+K+\varepsilon)}.$$

From this we have
\begin{eqnarray}
\frac{\partial^2 J_8^*(\hat{v}^*, \hat{p})}{\partial p^2}&=&
\frac{\partial^2 J^*(\hat{v}^*,\hat{v}_0^*,\hat{p})}{\partial p^2} \nonumber \\ &&+
\frac{\partial^2 J^*(\hat{v}^*,\hat{v}_0^*,\hat{p})}{\partial p \partial v_0^*}\frac{\partial \hat{v}_0^*}{\partial p}
\nonumber \\ &=& K-\frac{K^2}{(-\gamma \nabla^2+2 \hat{v}_0^*+K+\varepsilon)} \nonumber \\ &&
+\frac{2(\hat{v}^*+K\hat{p})K}{(-\gamma \nabla^2+2 \hat{v}_0^*+K+\varepsilon)^2}\frac{2\alpha K u_0}{(-\gamma \nabla^2+4 \alpha u_0^2+2 \hat{v}_0^*+K+\varepsilon)}.
\end{eqnarray}

Hence,
\begin{eqnarray}
\frac{\partial^2 J_8^*(\hat{v}^*, \hat{p})}{\partial p^2}&=&
K-\frac{K^2}{(-\gamma \nabla^2+2 \hat{v}_0^*+K+\varepsilon)} \nonumber \\ &&
+\frac{4 \alpha K^2 u_0^2 }{(-\gamma \nabla^2+2 \hat{v}_0^*+K+\varepsilon)}\frac{1}{(-\gamma \nabla^2+4 \alpha u_0^2+2 \hat{v}_0^*+K+\varepsilon)}
\end{eqnarray}
so that
\begin{eqnarray}
\frac{\partial^2 J_8^*(\hat{v}^*, \hat{p})}{\partial p^2}&=&
K-\frac{K^2}{(-\gamma \nabla^2+4\alpha u_0^2+2  \hat{v}_0^*+K+\varepsilon)} \nonumber \\ &=&
\frac{K (-\gamma \nabla^2+4\alpha u_0^2 +2 \hat{v}_0^*+\varepsilon)}{(-\gamma \nabla^2+4\alpha u_0^2+2  \hat{v}_0^*+K+\varepsilon)}  \nonumber \\ &=& K\frac{\delta^2 J(u_0)+\varepsilon}{(-\gamma \nabla^2+4\alpha u_0^2  +2\hat{v}_0^*+K+\varepsilon)} \nonumber \\ &>& \mathbf{0}.
\end{eqnarray}

Summarizing,
$$\frac{\partial^2 J_8^*(\hat{v}^*, \hat{p})}{\partial p^2}> \mathbf{0}.$$

Similarly,
\begin{eqnarray}
\frac{\partial^2 J_8^*(\hat{v}^*, \hat{p})}{\partial (v^*)^2}&=&
\frac{\partial^2 J^*(\hat{v}^*,\hat{v}_0^*,\hat{p})}{\partial (v^*)^2} \nonumber \\ &&+
\frac{\partial^2 J^*(\hat{v}^*,\hat{v}_0^*,\hat{p})}{\partial v^* \partial v_0^*}\frac{\partial \hat{v}_0^*}{\partial v^*}.
\end{eqnarray}

As above indicated,
$$\left(\frac{\hat{v}^*+K \hat{p}}{-\gamma \nabla^2+2 \hat{v}_0^*+K+\varepsilon}\right)^2-\frac{\hat{v}_0^*}{\alpha}
-\beta=\mathbf{0}$$

Hence, taking the variation in $v^*$ of such a last equation, we obtain
\begin{eqnarray}
&&\frac{2(\hat{v}^*+K \hat{p})}{(-\gamma \nabla^2+2 \hat{v}_0^*+K+\varepsilon)^2}
\nonumber \\ &&-4\frac{(\hat{v}^*+K \hat{p})^2}{(-\gamma \nabla^2+2 \hat{v}_0^*+K+\varepsilon)^3}\frac{\partial \hat{v}_0^*}{\partial v^*} \nonumber \\ &&-\frac{1}{\alpha}\frac{\partial \hat{v}_0^*}{\partial v^*}=\mathbf{0}.
\end{eqnarray}
so that
\begin{eqnarray}
&&\frac{2 u_0}{(-\gamma \nabla^2+2 \hat{v}_0^*+K+\varepsilon)}
\nonumber \\ &&-4\frac{(u_0)^2}{(-\gamma \nabla^2+2 \hat{v}_0^*+K+\varepsilon)}\frac{\partial v_0^*}{\partial v^*} \nonumber \\ &&-\frac{1}{\alpha}\frac{\partial \hat{v}_0^*}{\partial v^*}=\mathbf{0}.
\end{eqnarray}
so that
$$\frac{\partial \hat{v}_0^*}{\partial v^*}=\frac{2\alpha u_0}{(-\gamma \nabla^2+4 \alpha u_0^2+2 \hat{v}_0^*+K+\varepsilon)}.$$

From this we have
\begin{eqnarray}
\frac{\partial^2 J_8^*(\hat{v}^*, \hat{p})}{\partial (v^*)^2}&=&
\frac{\partial^2 J^*(\hat{v}^*,\hat{v}_0^*,\hat{p})}{\partial (v^*)^2} \nonumber \\ &&+
\frac{\partial^2 J^*(\hat{v}^*,\hat{v}_0^*,\hat{p})}{\partial v^* \partial v_0^*}\frac{\partial \hat{v}_0^*}{\partial v^*}
\nonumber \\ &=& \frac{1}{\varepsilon} -\frac{1}{(-\gamma \nabla^2+2 \hat{v}_0^*+K+\varepsilon)}\nonumber \\ &&
+\frac{2(\hat{v}^*+K\hat{p})}{[(-\gamma \nabla^2+2 \hat{v}_0^*+K+\varepsilon)^2]}\frac{2\alpha  u_0}{(-\gamma \nabla^2+4 \alpha u_0^2+2 \hat{v}_0^*+K+\varepsilon)}.
\end{eqnarray}

Hence,
\begin{eqnarray}
\frac{\partial^2 J_8^*(\hat{v}^*, \hat{p})}{\partial (v^*)^2}&=&
\frac{1}{\varepsilon}-\frac{1}{(-\gamma \nabla^2+2 \hat{v}_0^*+K+\varepsilon)} \nonumber \\ &&
+\frac{4 \alpha  u_0^2 }{(-\gamma \nabla^2+2 \hat{v}_0^*+K+\varepsilon)}\frac{1}{(-\gamma \nabla^2+4 \alpha u_0^2+2 \hat{v}_0^*+K+\varepsilon)}
\end{eqnarray}
so that
\begin{eqnarray}\label{a4b4}
\frac{\partial^2 J_8^*(\hat{v}^*, \hat{p})}{\partial (v^*)^2}&=&
\frac{1}{\varepsilon}-\frac{1}{(-\gamma \nabla^2+4\alpha u_0^2 +2 \hat{v}_0^*+K+\varepsilon)}
\nonumber \\ &>& \mathbf{0}.
\end{eqnarray}

Summarizing,
$$\frac{\partial^2 J_8^*(\hat{v}^*, \hat{p})}{\partial (v^*)^2}> \mathbf{0}.$$

Finally,
\begin{eqnarray}\label{a5b5}
\frac{\partial^2 J^*_3(\hat{v}^*)}{\partial (v^*)^2}&=&\frac{\partial^2  J^*(\hat{v}^*,\hat{v}_0^*, \hat{p})}{\partial (v^*)^2}
\nonumber \\ &&+\frac{\partial^2  J^*(\hat{v}^*,\hat{v}_0^*, \hat{p})}{\partial v^* \partial v^*_0}\frac{\partial \hat{v}^*_0}{\partial v^*} \nonumber \\ &&+\frac{\partial^2  J^*(\hat{v}^*,\hat{v}_0^*, \hat{p})}{\partial v^* \partial v^*_0}\frac{\partial \hat{v}^*_0}{\partial p}\frac{\partial \hat{p}}{\partial v^*} \nonumber \\ &&+ \frac{\partial^2  J^*(\hat{v}^*,\hat{v}_0^*, \hat{p})}{\partial v^* \partial p}\frac{\partial \hat{p}}{\partial v^*} \nonumber \\ &=& \frac{\partial^2 J_8^*(\hat{v}^*,\hat{v}_0^*, \hat{p})}{\partial (v^*)^2}\nonumber \\ &&+\frac{\partial^2  J^*(\hat{v}^*,\hat{v}_0^*, \hat{p})}{\partial v^* \partial v^*_0}\frac{\partial \hat{v}^*_0}{\partial p}\frac{\partial \hat{p}}{\partial v^*}\nonumber \\ &&+
 \frac{\partial^2  J^*(\hat{v}^*,\hat{v}_0^*, \hat{p})}{\partial v^* \partial p}\frac{\partial \hat{p}}{\partial v^*}
\end{eqnarray}
where from (\ref{a4b4}),
\begin{eqnarray}
\frac{\partial^2 J_8^*(\hat{v}^*, \hat{p})}{\partial (v^*)^2}&=&
\frac{1}{\varepsilon}-\frac{1}{(-\gamma \nabla^2+4\alpha u_0^2 +2 \hat{v}_0^*+K+\varepsilon)}
\nonumber \\ &>& \mathbf{0}.
\end{eqnarray}
From this and (\ref{a5b5}) we obtain
\begin{eqnarray}
\frac{\partial^2 J_3^*(\hat{v}^*)}{\partial (v^*)^2}&=&
\frac{1}{\varepsilon}-\frac{1}{(-\gamma \nabla^2+4\alpha u_0^2 +2 \hat{v}_0^*+K+\varepsilon)}
\nonumber \\ &&-\frac{4K \alpha u_0^2}{(-\gamma \nabla^2+4\alpha u_0^2 +2 \hat{v}_0^*+K+\varepsilon)} \frac{1}{(\delta^2J(u_0)+K+\varepsilon)}\frac{\partial \hat{p}}{\partial v^*}  \nonumber \\ &&
-\frac{K}{-\gamma \nabla^2+2\hat{v}_0^*+K+\varepsilon} \frac{\partial \hat{p}}{\partial v^*}.
\end{eqnarray}

However, from the variation of $J^*$ in $p$ we have
$$K\hat{p}-\frac{K(\hat{v}^*+K\hat{p})}{-\gamma \nabla^2+2\hat{v}_0^*+K+\varepsilon}=\mathbf{0},$$ so that
taking the variation in $v^*$ of this last equation, we get
\begin{eqnarray}&&K\frac{\partial \hat{p}}{\partial v^*}-\frac{K}{-\gamma \nabla^2+2\hat{v}_0^*+K+\varepsilon}
\nonumber \\ && -\frac{K^2}{-\gamma \nabla^2+2\hat{v}_0^*+K+\varepsilon}\frac{\partial \hat{p}}{\partial v^*}
\nonumber \\ &&+\frac{2(\hat{v}^*+K\hat{p})K}{(-\gamma \nabla^2+2\hat{v}_0^*+K+\varepsilon)^2}\left(\frac{\partial \hat{v}_0^*}{\partial v^*}+\frac{\partial \hat{v}_0^*}{\partial p}\frac{\partial \hat{p}}{\partial v^*}\right)=\mathbf{0},
\end{eqnarray}
so that
\begin{eqnarray}&&K\frac{\partial \hat{p}}{\partial v^*}-\frac{K}{-\gamma \nabla^2+2\hat{v}_0^*+K+\varepsilon}
\nonumber \\ && -\frac{K^2}{(-\gamma \nabla^2+2\hat{v}_0^*+K+\varepsilon)}\frac{\partial \hat{p}}{\partial v^*}
\nonumber \\ &&+\frac{4\alpha K^2 u_0^2}{(-\gamma \nabla^2+2\hat{v}_0^*+K+\varepsilon)}\frac{1}{(\delta^2J(u_0)+K+\varepsilon)}\frac{\partial \hat{p}}{\partial v^*}
\nonumber \\ && +\frac{4\alpha K u_0^2}{(-\gamma \nabla^2+2\hat{v}_0^*+K+\varepsilon)}\frac{1}{(\delta^2J(u_0)+K+\varepsilon)}\nonumber \\ &=& 0,
\end{eqnarray}
 Summarizing,
$$\frac{\partial \hat{p}}{\partial v^*}=\frac{1}{(\delta^2J(u_0)+\varepsilon)}$$
so that, considering that $K\gg \varepsilon$, we may obtain
\begin{eqnarray}\label{a10b10}
\frac{\partial^2 J_3^*(\hat{v}^*)}{\partial (v^*)^2}&=&
\frac{1}{\varepsilon} -\frac{1}{(\delta^2J(u_0)+K+\varepsilon)}\nonumber \\ &&+\frac{4K\alpha u_0^2}{(-\gamma \nabla^2 +2 \hat{v}_0^*+K+\varepsilon)} \frac{1}{(\delta^2J(u_0)+K+\varepsilon)}\frac{1}{(\delta^2J(u_0)+\varepsilon)}  \nonumber \\ &&
-\frac{K}{(-\gamma \nabla^2+2 \hat{v}_0^*+K+\varepsilon)}\frac{1}{(\delta^2J(u_0)+\varepsilon)}
\nonumber \\ &=&\frac{1}{\varepsilon}
-\frac{1}{(\delta^2 J(u_0)+K+\varepsilon)}\nonumber \\ &&-\frac{K}{(\delta^2J(u_0)+K+\varepsilon)} \frac{1}{(\delta^2J(u_0)+\varepsilon)}
\nonumber \\ &=& \frac{1}{\varepsilon}
-\frac{1}{(\delta^2 J(u_0)+\varepsilon)}\nonumber \\ &=&\mathcal{O}\left(\frac{1}{\varepsilon}\right)
\nonumber \\ &>& \mathbf{0},
\end{eqnarray}
in $B_{r_3}(\hat{v})$ for an appropriate not relabeled $r_3>0$, for a sufficiently small $\varepsilon>0.$

From such results, we may infer that there exist not relabeled $r_1,r_2,r_3>0$ such that
\begin{eqnarray}
J(u_0)&=& \inf_{u \in B_{r_1}(u_0)} J(u) \nonumber \\ &=&
\inf_{v^* \in B_{r_3}(\hat{v}^*)}\left\{\inf_{p \in B_{r_2}(\hat{p})}\left\{\sup_{v_0^* \in B^*}J^*(v^*,v_0^*,p)\right\}\right\}
\nonumber \\ &=& J^*(\hat{v}^*,\hat{v}_0^*, \hat{p}).
\end{eqnarray}

Moreover,
$$\delta J^*_3(\hat{v}^*)=\mathbf{0}$$
$$\delta^2 J_3^*(\hat{v}^*) > \mathbf{0}$$ so that
\begin{eqnarray}
J(u_0)&=& \inf_{u \in B_{r_1}(u_0)}J(u) \nonumber \\ &=& \inf_{v^* \in B_{r_3}(\hat{v}^*)} J_3^*(v^*) \nonumber \\
&=& J_3^*(\hat{v}^*).
\end{eqnarray}
The proof of the item (\ref{a1b1}) is complete.

For the item (\ref{a2b2}), suppose $u_0 \in U$ is such that $\delta J(u_0)=\mathbf{0}$ and
$$\delta^2J(u_0)< \mathbf{0}.$$

Similarly as obtained above we may get
$$\frac{\partial J_8^*(\hat{v}^*,\hat{p})}{\partial p^2}<\mathbf{0},$$
and
$$\frac{\partial^2 J_5^*(\hat{v}^*)}{\partial (v^*)^2}>\mathbf{0}.$$
Hence, there exist not relabeled real constants $r_1,r_2,r_3>0$ such that
\begin{eqnarray}
J(u_0)&=& \sup_{u \in B_{r_1}(u_0)} J(u) \nonumber \\ &=&
\inf_{v^* \in B_{r_3}(\hat{v}^*)}\left\{\sup_{p \in B_{r_2}(\hat{p})}\left\{\sup_{v_0^* \in B^*}J^*(v^*,v_0^*,p)
\right\} \right\}
\nonumber \\ &=& J^*(\hat{v}^*,\hat{v}_0^*, \hat{p}).
\end{eqnarray}

Moreover,
$$\delta J^*_5(\hat{v}^*)=\mathbf{0}$$
$$\delta^2 J_5^*(\hat{v}^*) >\mathbf{0}$$ so that
\begin{eqnarray}
J(u_0)&=& \sup_{u \in B_{r_1}(u_0)}J(u) \nonumber \\ &=& \inf_{v^* \in B_{r_3}(\hat{v}^*)} J^*_5(v^*) \nonumber \\
&=& J_5^*(\hat{v}^*).
\end{eqnarray}
The proof of the item (\ref{a2b2}) is complete.
For the item (\ref{a3b3}) we recall that
$$J_7^*:Y^* \times Y \rightarrow \mathbb{R}$$ is defined by
$$J_7^*(v^*,p)=\sup_{v_0^* \in A^+ \cap B^*} J^*(v^*,v_0^*,p).$$

Observe that through a direct computation we may obtain that the Hessian $$\left\{\frac{\partial^2 J^*(v^*,v_0^*,p)}{\partial v^* \partial p}\right\}$$ is positive definite in $Y^*\times (A^+\cap B^*) \times Y$ so that $J^*_7$ is convex  as the supremum of a family of convex functionals. Summarizing, we have got $$\delta^2 J_7^*(v^*,p)>\mathbf{0}$$ in $Y^* \times Y.$

From these results we may obtain
\begin{eqnarray}\label{a8b8}
J_7^*(\hat{v}^*,\hat{p})&=&\inf_{(v^*,p) \in Y^* \times Y} J_9^*(v^*,p)
\nonumber \\ &=&\inf_{(v^*,p) \in Y^* \times Y}\left\{ \sup_{v_0^* \in A^+\cap B^*} J^*(v^*,v_0^*,p)\right\}\nonumber \\ &=& J^*(\hat{v}^*, \hat{v}_0^*, \hat{p}) \nonumber \\ &=& J(u_0).
\end{eqnarray}

On the other hand
\begin{eqnarray}
J(u_0)&=& J^*(\hat{v}^*,\hat{v}_0^*,\hat{p}) \nonumber \\ &=&
-G^*(\hat{v}^*,\hat{v}_0^*, \hat{p})-F^*(\hat{v}^*)+H(\hat{p})
\nonumber \\ &=&  \inf_{(v^*,p) \in Y^* \times Y}\left\{\sup_{v_0^* \in A^+ \cap B^*} J^*(v^*,v_0^*,p)\right\}
\nonumber \\ &\leq& \left\{\sup_{v_0^* \in A^+\cap B^+}
\left\{ \frac{\gamma}{2}\int_\Omega \nabla u \cdot \nabla u\;dx+\int_\Omega v_0^* u^2\;dx \right.\right.
\nonumber \\ && -\frac{1}{2\alpha}\int_\Omega (v_0^*)^2\;dx-\beta \int_\Omega v_0^*\;dx \nonumber \\ &&
\left.\left.\frac{K+\varepsilon}{2}\int_\Omega u^2\;dx-\int_{\Omega} K p u \;dx+ \frac{K}{2}\int_{\Omega} p^2\;dx
-\langle u,v^* \rangle_{L^2}+F^*(v^*) \right\}\right\} \nonumber \\ &\leq&\left\{\sup_{v_0^* \in Y^*}
\left\{ \frac{\gamma}{2}\int_\Omega \nabla u \cdot \nabla u\;dx+\int_\Omega v_0^* u^2\;dx \right.\right.
\nonumber \\ && -\frac{1}{2\alpha}\int_\Omega (v_0^*)^2\;dx-\beta \int_\Omega v_0^*\;dx \nonumber \\ &&
\left.\left.\frac{K+\varepsilon}{2}\int_\Omega u^2\;dx-\int_{\Omega} K p u \;dx+ \frac{K}{2}\int_{\Omega} p^2\;dx
-\langle u,v^* \rangle_{L^2}+F^*(v^*) \right\}\right\} ,
\end{eqnarray}
$ \forall u \in U, p \in Y,\;v^* \in Y^*.$

From this, in particular for $v^*=\varepsilon u+f$ we may infer that
\begin{eqnarray}J(u_0)&\leq& \frac{\gamma}{2}\int_\Omega \nabla u \cdot \nabla u\;dx  \nonumber \\ &&
+\frac{\alpha}{2}\int_\Omega (u^2-\beta)^2\;dx+\frac{K}{2}\int_\Omega (u-p)^2\;dx
\nonumber \\ &&-\langle u,f \rangle_{L^2} \nonumber \\ &=& J(u,p),\; \forall u \in U, p \in Y.
\end{eqnarray}

Consequently, from  such a result and (\ref{a8b8}) we may infer that
\begin{eqnarray}
J(u_0)&=& \inf_{u \in U} J(u) \nonumber \\ &=&
\inf_{(v^*,p) \in Y^* \times Y}\left\{\sup_{v_0^* \in A^+ \cap B^*}J^*(v^*,v_0^*,p)\right\}
\nonumber \\ &=& J^*(\hat{v}^*,\hat{v}_0^*, \hat{p}).
\end{eqnarray}

Moreover, considering as previously indicated, that $J_7^*:C^* \rightarrow \mathbb{R}$ is defined by
$$J_7^*(v^*,p)=\left\{\sup_{v_0^* \in A^+ \cap B^*}J^*(v^*,v_0^*,p)\right\}$$
we get also
$$\delta J_7^*(\hat{v}^*,\hat{p})=\mathbf{0}$$
$$\delta^2 J_7^*(\hat{v}^*,\hat{p}) > \mathbf{0}$$ so that
\begin{eqnarray}
J(u_0)&=& \inf_{u \in U}J(u) \nonumber \\ &=& \inf_{(v^*,p) \in Y^* \times Y} J_7^*(v^*,p) \nonumber \\
&=& J_7^*(\hat{v}^*,\hat{p}).
\end{eqnarray}

The proof is complete.
\end{proof}
\section{A criterion for global optimality}

In this section we present a new concerning optimality criterion.

\begin{thm}\label{a12b12} Let $\Omega \subset \mathbb{R}^3$ be an open, bounded and connected set with a regular
(Lipschitzian) boundary denoted by $\partial \Omega.$

Consider the functionals $\hat{J}:U \times Y \rightarrow \mathbb{R}$ and $J:U \rightarrow \mathbb{R}$  where
\begin{eqnarray}J(u,p)&=&\frac{\gamma}{2}\int_\Omega \nabla u \cdot \nabla u \;dx
+ \frac{\alpha}{2}\int_\Omega (u^2-\beta)^2\;dx \nonumber \\ &&+\frac{K}{2}\int_\Omega (u-p)^2\;dx-\langle u,f \rangle_{L^2},
\end{eqnarray}
and $$J(u)=\hat{J}(u,u),\; \forall u \in U.$$
where $\alpha>0,$ $\beta>0$, $\gamma>0$ and $f \in C^1(\overline{\Omega}).$

Assume either $$f(x) \geq 0,\; \forall x \in \overline{\Omega}$$ or
$$f(x) \leq 0,\; \forall x \in \overline{\Omega}.$$

Suppose also, in a matrix sense $$-\gamma \nabla^2-2\alpha\beta \leq \mathbf{0},$$ assuming from now and on a finite dimensional approximation for the model in question, in a finite elements or finite differences context, even though the spaces, functionals and operators have not been relabeled.

Moreover define,
$$A^+=\{u \in U\;:\; u f \geq 0, \text{ in } \Omega\}$$

and
$$B^+=\{ u \in U\;:\; \delta^2J(u) \geq \mathbf{0}\}.$$

Under such hypotheses,
$$\inf_{u \in U} J(u)=\inf_{u \in A^+} J(u).$$

Furthermore, $$A^+ \cap B^+$$ is convex.
\end{thm}
\begin{proof} Define $$\alpha_1=\inf_{u \in U} J(u).$$

Let $\varepsilon>0$.

Thus we may obtain $u_\varepsilon \in U$ such that

$$\alpha_1 \leq J(u_\varepsilon)< \alpha_1+\varepsilon.$$

Define $v_\varepsilon \in A^+$ by

Define  \begin{equation}
v_\varepsilon(x)=\left \{
\begin{array}{ll}
 u_\varepsilon(x), &  \text{ if }\; u_\varepsilon(x)f(x) \geq 0,
 \\
 -u_\varepsilon(x), &  \text{ if }\; u_\varepsilon(x)f(x) < 0,
  \end{array} \right.\end{equation}
$\forall x \in \overline{\Omega}.$

Observe that \begin{eqnarray} J(v_\varepsilon)&=& \frac{\gamma}{2}\int_\Omega \nabla v_\varepsilon \cdot \nabla v_\varepsilon \;dx+\frac{\alpha}{2}\int_\Omega
(v_\varepsilon^2-\beta)^2\;dx \nonumber \\ && -\langle v_\varepsilon,f \rangle_{L^2}
\nonumber \\ &\leq& \frac{\gamma}{2}\int_\Omega \nabla u_\varepsilon \cdot \nabla u_\varepsilon \;dx+\frac{\alpha}{2}\int_\Omega
(u_\varepsilon^2-\beta)^2\;dx \nonumber \\ && -\langle u_\varepsilon,f \rangle_{L^2} \nonumber \\ &=&
J(u_\varepsilon).\end{eqnarray}

Hence $$\alpha_1 \leq J(v_\varepsilon) \leq J(u_\varepsilon) < \alpha_1+\varepsilon.$$

From this, since $v_\varepsilon \in A^+$,  we obtain
$$\alpha_1 \leq \inf_{u \in A^+}J(u) < \alpha_1+\varepsilon.$$

Since $\varepsilon>0$ is arbitrary, we may infer that
$$\inf_{u \in U}J(u)=\alpha_1=\inf_{ u \in A^+}J(u).$$

Finally, observe also that
$$\delta^2J(u)=-\gamma \nabla^2+6\alpha u^2-2\alpha\beta \geq \mathbf{0},$$ if, and only if
$$H(u)\geq \mathbf{0},$$
where
$$H(u)=\sqrt{6\alpha}|u|-\sqrt{\gamma\nabla^2+2\alpha\beta}\geq\mathbf{0}.$$
Hence, if $u_1,u_2 \in A^+ \cap B^+$ and $\lambda \in [0,1]$, then $$H(|u_1|)\geq \mathbf{0},$$
$$H(|u_2|)\geq \mathbf{0}$$ and also since $$\text{ sign } u_1=\text{ sign }u_2, \text{ in } \Omega,$$ we get
$$|\lambda u_1+(1-\lambda) u_2|=\lambda |u_1|+(1-\lambda)|u_2|,$$
so that,  $$H(|\lambda u_1+(1-\lambda)u_2|)=H(\lambda |u_1|+(1-\lambda)|u_2|)=\lambda H(|u_1|)+(1-\lambda)H(|u_2|)\geq \mathbf{0}$$ and thus, $$\delta^2J(\lambda u_1+(1-\lambda)u_2)\geq\mathbf{0}.$$

From this, we may infer that $A^+ \cap B^+$ is convex.

The proof is complete.
\end{proof}

\section{ Another related duality principle}

In this subsection we develop a duality principle concerning the last optimality criterion established.
 \begin{thm}\label{a14b14} Let $\Omega \subset \mathbb{R}^3$ be an open, bounded and connected set with a regular
(Lipschitzian) boundary denoted by $\partial \Omega.$

Consider the functionals $\hat{J}:U \times Y \rightarrow \mathbb{R}$ and $J:U \rightarrow \mathbb{R}$  where
\begin{eqnarray}J(u,p)&=&\frac{\gamma}{2}\int_\Omega \nabla u \cdot \nabla u \;dx
+ \frac{\alpha}{2}\int_\Omega (u^2-\beta)^2\;dx \nonumber \\ &&+\frac{K}{2}\int_\Omega (u-p)^2\;dx-\langle u,f \rangle_{L^2},
\end{eqnarray}
and $$J(u)=\hat{J}(u,u),\; \forall u \in U,$$
where $\alpha,\beta,\gamma$ are positive real constants,
$U=W_0^{1,2}(\Omega)$, $f \in C^1(\overline{\Omega})$ and we also denote $Y=Y^*=L^2(\Omega).$

Here we assume $$-\gamma \nabla^2-2 \alpha \beta \leq  \mathbf{0}$$
in an appropriate matrix sense considering, as above indicated, a finite dimensional not relabeled model approximation, in a finite differences or finite elements context.

Assume also either $$f(x)\geq 0,\; \forall x \in \overline{\Omega}$$ or $$f(x)\leq 0,\;\forall x \in \overline{\Omega}.$$

Define $G:U \times Y \rightarrow \mathbb{R}$ by
\begin{eqnarray}G(u,p)&=&\frac{\gamma}{2}\int_\Omega \nabla u \cdot \nabla u\;dx+\frac{\alpha}{2}\int_\Omega (u^2-\beta)^2\;dx \nonumber \\ &&+\frac{K+\varepsilon}{2} \int_\Omega u^2\;dx-\langle u,Kp\rangle_{L^2}\end{eqnarray}
 $F: U \rightarrow \mathbb{R}$ by
$$F(u)=\frac{\varepsilon}{2}\int_\Omega u^2\;dx-\langle u,f \rangle_{L^2}$$
and $H:Y \rightarrow \mathbb{R}$ by $$H(p)=\frac{K}{2}\int_\Omega p^2\;dx.$$
so that
$$\hat{J}(u,p)=G(u,p)-F(u)+H(p)$$

Furthermore, define $G^*: Y^* \times Y \rightarrow \mathbb{R}$ by
$$G^*(v^*+Kp)=\sup_{ u \in U}\{\langle u,v^* \rangle_{L^2} -G(u,p)\},$$
$F^*:Y^* \rightarrow \mathbb{R}$ by
\begin{eqnarray}F^*(v^*)&=&\sup_{u \in U}\{\langle u,v^*\rangle_{L^2}-F(u)\} \nonumber \\ &=& \frac{1}{2\varepsilon}\int_\Omega (v^*-f)^2\;dx.\end{eqnarray}
and $J^*:Y^* \times Y \rightarrow \mathbb{R}$ as
$$J^*(v^*,p)=-G^*(v^*+Kp)+F^*(v^*)+H(p).$$

Define also,
$$A^+=\{u \in U \;:\; u f \geq 0, \text{ in } \overline{\Omega}\},$$
$$B^+=\{u \in U\;:\; \delta^2J(u)\geq\mathbf{0}\},$$
$$E=A^+ \cap B^+,$$
Moreover, define
$$\hat{v}_0^*=\alpha (u_0^2-\beta),$$
$$\hat{v}^*=\varepsilon u_0+f,$$
$$\hat{p}=u_0,$$
and assume $u_0 \in U$ is such that $\delta J(u_0)=\mathbf{0},$ and
 $$u_0 \in E,$$
Under such hypothesis, assuming also $\hat{v}_0^* \in B^*$ we have

\begin{eqnarray}
J(u_0)&=& \inf_{u \in E}J(u) \nonumber \\ &=& \inf_{u \in U}J(u)
\nonumber \\ &=&  \inf_{(v^* ,p) \in Y^* \times Y} J^*(v^*,p)
\nonumber \\ &=& J^*(\hat{v}^*,\hat{p}).
\end{eqnarray}
\end{thm}
\begin{proof} Define $$\alpha_1=\inf_{u \in U} J(u).$$

Hence
\begin{eqnarray}
\alpha_1 &\leq& J(u,p) \nonumber \\ &=& G(u,p)-F(u)+H(p)\nonumber \\ &\leq&
-\langle u,v^*\rangle_{L^2}+G(u,p)+H(p) \nonumber \\ && +\sup_{u \in U}\{\langle u,v^*\rangle_{L^2}-F(u)\}
\nonumber \\ &=& -\langle u,v^*\rangle_{L^2}+G(u,p)+H(p) +F^*(v^*)\end{eqnarray}
$\forall u \in U, v^* \in Y^*,\; p \in Y.$

Thus,
\begin{eqnarray}
\alpha_1  &\leq&
\inf_{u \in U}\{-\langle u,v^*\rangle_{L^2}+G(u,p)\}+H(p)+F^*(v^*)
\nonumber \\ &=& \-G^*(v^*+Kp)+F^*(v^*)+H(p)\end{eqnarray}
$\forall v^* \in Y^*,\; p \in Y.$
 Summarizing  \begin{equation}\label{a18b18}\alpha_1=\inf_{u \in U} J(u) \leq \inf_{(v^*,p) \in Y^*\times Y}J^*(v^*,p).\end{equation}

 From Theorem \ref{a12b12} we have that $$\alpha_1=J(u_0)=\inf_{u \in U} J(u)=\inf_{u \in E}J(u).$$

 Similarly as in the proof of Theorem \ref{a11b11} we may obtain
 $$\delta J^*(\hat{v}^*,p)=\mathbf{0}$$ and $$J^*(\hat{v}^*,\hat{p})=\hat{J}(u_0,\hat{p})=\hat{J}(u_0,u_0)=J(u_0).$$

 From this and (\ref{a18b18}) we may infer that
 \begin{eqnarray}
J(u_0)&=& \inf_{u \in E}J(u) \nonumber \\ &=& \inf_{u \in U}J(u)
\nonumber \\ &=&  \inf_{(v^* ,p) \in Y^* \times Y} J^*(v^*,p)
\nonumber \\ &=& J^*(\hat{v}^*,\hat{p}).
\end{eqnarray}

The proof is complete.
\end{proof}
\section{A  convex dual variational formulation}

Let $\Omega$ be an open, bounded and connected set with a regular (Lipschitzian) boundary denoted by $\partial \Omega$.

In this section we define $G:U \rightarrow \mathbb{R}$ by

\begin{eqnarray}
G(u)&=&\frac{\gamma}{2}\int_\Omega \nabla u \cdot \nabla u \;dx \nonumber \\ &&
+\frac{K}{2}\int_\Omega u^2\;dx- \langle u,f\rangle_{L^2},
\end{eqnarray}
and
$F:U \rightarrow \mathbb{R}$ by
\begin{eqnarray}
F(u)&=& -\frac{\alpha}{2}\int_\Omega (u^2-\beta)^2\;dx
\nonumber \\ &&+\frac{K}{2}\int_\Omega u^2\;dx,
\end{eqnarray}
where
$\alpha,\beta,\gamma>0,\; f \in L^2(\Omega) \text{ and } U=W_0^{1,2}(\Omega).$

Moreover we define
$$U_1=\{ u \in U\;:\;\|u\|_\infty \leq \sqrt[4]{K}\},$$
where $K>0$ is such that $G$ and $F$ are convex in $U_1$.

Define also
$$B^+=\{u \in U \;:\; \delta^2J(u) \geq \mathbf{0}\}$$
where $J:U \rightarrow \mathbb{R}$ is given by
\begin{eqnarray}
J(u)&=& G(u)-F(u) \nonumber \\ &=&
\frac{\gamma}{2}\int_\Omega \nabla u \cdot \nabla u \;dx \nonumber \\ &&
+\frac{\alpha}{2}\int_\Omega (u^2-\beta)^2\;dx- \langle u,f\rangle_{L^2}.
\end{eqnarray}

Finally, define $$D^*=\{v^* \in Y^*=L^2(\Omega)\;:\; \|v^*\|_\infty  \leq 3K\},$$
$G^*:D^*\rightarrow \mathbb{R}$ by
\begin{eqnarray}
G^*(v^*)&=& \sup_{ u \in U_1}\{ \langle u,v^* \rangle_{L^2}-G(u)\} \nonumber \\ &=&
\frac{1}{2}\int_\Omega \frac{(v^*+f)^2}{(-\gamma \nabla^2+K)}\;dx,
\end{eqnarray}
$F^*:D^* \rightarrow \mathbb{R}$ by
$$F^*(v^*)=\sup_{u \in U_1 \cap B^+} \{\langle u,v^* \rangle_{L^2}-F(u) \}$$
and $J^*:D^* \rightarrow \mathbb{R}$ by
$$J^*(v^*)=-G^*(v^*)+F^*(v^*)$$
Assume now either $$f(x)>0, \forall x \in \overline{\Omega}$$
or
$$f(x)<0, \forall x \in \overline{\Omega}.$$

Define $$A^+=\{u \in U_1\;:\; u\;f >0 \text{ in } \Omega\}$$
and define also $$D_1^*=\left\{v^* \in D^* \;:\; \hat{u}=\frac{\partial F_1^*(v^*)}{\partial v^*} \in A^+\right\},$$
where $$F_1^*(v^*)=\sup_{u \in U}\{ \langle u,v^*\rangle_{L^2}-F(u)\}.$$
\begin{thm} Under the hypotheses state above $J^*$ is convex on $D_1^*$.
\end{thm}
\begin{proof}
Let $v^* \in D^*_1$.

Thus
\begin{eqnarray}
F^*(v^*)&=& \sup_{u \in U_1 \cap B^+}\{\langle u,v^* \rangle_{L^2}-F(u) \}
\nonumber \\ &=& \sup_{u \in U}\{\langle u,v^* \rangle_{L^2}-F(u) \nonumber \\ &&
+\gamma\int_\Omega\nabla \varphi \cdot \nabla \varphi \;dx+6\alpha\int_\Omega u^2\varphi^2 \;dx
\nonumber \\ &&-2\alpha\beta \int_\Omega \varphi^2\;dx-\int_\Omega \varphi_1^2(u^2-\sqrt[2]{K})\;dx\},
\end{eqnarray}
for some appropriate Lagrange multipliers $(\varphi,\varphi_1) \in W^{1,2}(\Omega) \times L^2(\Omega).$

The last supremum is attained for some $\hat{u} \in U$ such that
$$v^*-\frac{\partial F(\hat{u})}{\partial u} + \varphi^2 (12)\alpha \hat{u}-\varphi_1^2 (2 \hat{u})=0, \text{ in } \Omega.$$

Taking the variation in $v^*$ in this last equation, we get
\begin{eqnarray}
&&\frac{\partial v^*}{\partial v^*}-\frac{\partial^2 F(\hat{u})}{\partial u^2}\frac{\partial \hat{u}}{\partial v^*}
\nonumber \\ && +\varphi^2(12 \alpha)\frac{\partial \hat{u}}{\partial v^*}-\varphi_1^22 \frac{\partial \hat{u}}{\partial v^*} \nonumber \\ &&+24\alpha \varphi \partial_{v^*} \varphi \hat{u}-12 \varphi_1 \partial_{v^*} \varphi_1 \hat{u}=0, \text{ in } \Omega.\end{eqnarray}

On the other hand we have the following necessary condition to be satisfied
\begin{eqnarray}
&&\gamma\int_\Omega\nabla \varphi \cdot \nabla \varphi \;dx+6\alpha\int_\Omega u^2\varphi^2 \;dx
\nonumber \\ &&-2\alpha\beta \int_\Omega \varphi^2\;dx=0,
\end{eqnarray}
so that
$$2\varphi \partial_u \varphi(-\gamma\nabla^2+6\alpha \hat{u}^2-2\alpha \beta)+\varphi^2(12\alpha \hat{u})=0$$
$\text{ in } \Omega$ so that
$$\varphi^2(12\alpha \hat{u})=0
\text{ in } \Omega.$$

And also, we must have
$$\int_\Omega \varphi_1^2(u^2-\sqrt[2]{K})\;dx=0,$$
so that
$$2 \varphi_1 \partial_u \varphi_1(\hat{u}^2-\sqrt[2]{K})+2\varphi_1^2 \hat{u}=0, \text{ in } \Omega,$$ and thus
$$2\varphi_1^2 \hat{u}=0, \text{ in } \Omega.$$

Since $v^* \in D^*$ we may assume $\varphi_1=0$ and thus
\begin{eqnarray}
&&\frac{\partial v^*}{\partial v^*}-\frac{\partial^2 F(\hat{u})}{\partial u^2}\frac{\partial \hat{u}}{\partial v^*}
\nonumber \\ && +\varphi^2(12 \alpha)\frac{\partial \hat{u}}{\partial v^*}-\varphi_1^22 \frac{\partial \hat{u}}{\partial v^*} \nonumber \\ &&+24\alpha \varphi \partial_{v^*} \varphi \hat{u}-12 \varphi_1 \partial_{v^*} \varphi_1 \hat{u} \nonumber \\ &=& 1-\frac{\partial^2 F(\hat{u})}{\partial u^2}\frac{\partial \hat{u}}{\partial v^*}
\nonumber \\ && +\varphi^2(12 \alpha)\frac{\partial \hat{u}}{\partial v^*}=0, \text{ in } \Omega.\end{eqnarray}

Therefore,
\begin{eqnarray}
\frac{\partial \hat{u}}{\partial v^*} &=& \frac{1}{\frac{\partial^2F(\hat{u})}{\partial u^2}-12\alpha \varphi^2}
\nonumber \\ &=& \frac{1}{-6 \alpha \hat{u}^2+2\alpha \beta-12\alpha \varphi^2+K} \nonumber \\ &>& \mathbf{0}.
\end{eqnarray}

On the other hand

\begin{eqnarray}
F^*(v^*)&=& \langle \hat{u},v^* \rangle_{L^2}-F(\hat{u}) \nonumber \\ &&
+\gamma\int_\Omega\nabla \varphi \cdot \nabla \varphi \;dx+\int_\Omega 6\alpha \hat{u}^2\varphi^2 \;dx
\nonumber \\ &&-2\alpha\beta \int_\Omega \varphi^2\;dx-\int_\Omega \varphi_1^2(\hat{u}^2-\sqrt[2]{K})\;dx,
\end{eqnarray}
Hence
\begin{eqnarray}
\frac{\partial F^*(v^*)}{\partial v^*} &=& \hat{u}+\left(v^*-\frac{\partial F(\hat{u})}{\partial u} \right.
\nonumber \\ &&\left. +\varphi^212 \alpha \hat{u}-\varphi_1^2(2\hat{u})\right) \frac{\partial \hat{u}}{\partial v^*}
\nonumber \\ && 2\varphi  \partial_{v^*} \varphi (-\gamma \nabla^2+6\alpha \hat{u}^2-2\alpha\beta)
\nonumber \\ &&-2\varphi_1 \partial_{v^*} \varphi_1 (\hat{u}^2-\sqrt[2]{K}) \nonumber \\ &=& \hat{u}.\end{eqnarray}
Therefore, we may infer that
\begin{eqnarray}
\frac{\partial^2 F^*(v^*)}{\partial (v^*)^2}&=&\frac{\partial \hat{u}}{\partial v^*}  \nonumber \\ &=& \frac{1}{-6 \alpha \hat{u}^2+2\alpha \beta-12\alpha \varphi^2+K} \nonumber \\ &>& \mathbf{0}.
\end{eqnarray}

Thus,

\begin{eqnarray}
\frac{\partial^2 J^*(v^*)}{\partial (v^*)^2}&=&-\frac{\partial^2 G^*(v^*)}{\partial (v^*)^2}+\frac{\partial^2 F^*(v^*)}{\partial (v^*)^2}  \nonumber \\ &=& -\frac{1}{-\gamma \nabla^2+K}+ \frac{1}{-6 \alpha \hat{u}^2+2\alpha \beta-12\alpha \varphi^2+K}\nonumber \\ &=&\frac{ -\gamma \nabla^2+6\alpha \hat{u}^2-2\alpha \beta+12\alpha \varphi^2}{(-\gamma\nabla^2+K)(-6 \alpha \hat{u}^2+2\alpha \beta-12\alpha \varphi^2+K)}
\nonumber \\ &=&\frac{\delta^2 J(\hat{u})+12 \alpha \varphi^2}{(-\gamma\nabla^2+K)(-6 \alpha \hat{u}^2+2\alpha \beta-12\alpha \varphi^2+K)}  \nonumber \\ &>& \mathbf{0}, \; \forall v^* \in D_1^*
\end{eqnarray}

From this we may infer that $J^*$ is convex on $D_1^*$.

\end{proof}
In the next lines we present our main result.

\begin{thm} Let $\hat{v}^* \in D^*_1$ be such that
$$\delta J^*(\hat{v}^*)= \mathbf{0}.$$

Assume either $$f(x)>0, \forall x \in \overline{\Omega}$$
or
$$f(x)<0, \forall x \in \overline{\Omega}.$$

Define $$A^+=\{u \in U_1\;:\; u\;f >0 \text{ in } \Omega\}$$
and
$$u_0=(-\gamma \nabla^2+K)^{-1} \hat{v}^*$$

Assume also $$u_0 \in A^+ \cap B^+$$
and recall that $$D_1^*=\left\{v^* \in D^* \;:\; \hat{u}=\frac{\partial F_1^*(v^*)}{\partial v^*} \in A^+\right\}.$$

Under such hypotheses

\begin{eqnarray}
J(u_0)&=& \inf_{u \in U_1} J(u) \nonumber \\ &=& \inf_{v^* \in D^*_1} J^*(v^*) \nonumber \\ &=& J^*(\hat{v}^*).
\end{eqnarray}

\end{thm}
\begin{proof}
From the last theorem $J^*$ is convex in $D_1^*$ so that
$$J^*(\hat{v}^*)=\inf_{v^* \in D_1^*} J^*(v^*).$$

Therefore,
\begin{eqnarray}
J^*(\hat{v}^*)&\leq& J^*(v^*) \nonumber \\ &=& -G^*(v^*)+F^*(v^*) \nonumber \\ &\leq&
-\langle u,v^* \rangle_{L^2}+G(u)+F_1^*(v^*)\; \forall u \in U_1, \; v^* \in D^*_1.
\end{eqnarray}
Hence
\begin{eqnarray}
J^*(\hat{v}^*)&\leq&
\inf_{v^* \in D_1^*}\{-\langle u,v^* \rangle_{L^2}+G(u)+F_1^*(v^*)\} \nonumber \\ &=& G(u)-F(u) \nonumber \\ &=& J(u),\; \forall u \in A^+.
\end{eqnarray}
Similarly as in the previous theorems proofs we may obtain
 \begin{equation}\label{uk10A}\inf_{u \in U_1} J(u)=\inf_{ u \in A^+}J(u)  \geq \inf_{v^* \in D_1^*}J^*(v^*).\end{equation}

On the other hand, also similarly as the proofs of the previous theorems we may obtain

$$\delta J(u_0)=\mathbf{0}$$ and $$J(u_0)=J^*(\hat{v}^*).$$

From this and (\ref{uk10A}) we may infer that
\begin{eqnarray}
J(u_0)&=& \inf_{u \in U_1} J(u) \nonumber \\ &=& \inf_{v^* \in D^*_1} J^*(v^*) \nonumber \\ &=& J^*(\hat{v}^*).
\end{eqnarray}
The proof is complete.
\end{proof}

\section{ A final dual variational formulation}

This final duality principle is summarized by the following theorem.

\begin{thm} Let $U,Y$ be a Banach spaces such that $Y=Y^*$ and let $\Lambda :U \rightarrow Y$ be a bounded linear operator.

Consider the functional $J:U \rightarrow \mathbb{R}$ expressed by
$$J(u)=G_K(\Lambda u)-F(\Lambda u)-\langle u,f \rangle_U,$$ where $G_K:Y \rightarrow \mathbb{R}$ is defined by
$G_K(\Lambda u)= G(\Lambda u)+\frac{K}{2}\langle \Lambda u,\Lambda u \rangle_Y$ where $G:Y \rightarrow \mathbb{R}$ is a  coercive,  Fr\'{e}chet differentiable and possibly non-convex functional. Moreover $f \in U^*$ and $F:Y \rightarrow \mathbb{R}$ is such that
$$F(\Lambda u)=\frac{K}{2}\langle \Lambda u,\Lambda u \rangle_Y,$$
so that $$J(u)=G(\Lambda u)-\langle u,f \rangle_U.$$
Assume $$\inf_{u \in U} J(u)= \alpha \in \mathbb{R}$$ and  $K>0$ is such that $G_K$ is convex.

Define the polar functionals $G^*_K:Y^* \rightarrow \mathbb{R}$ and $F^*:Y^* \rightarrow \mathbb{R}$ by

$$G_K^*(v^*+z^*)= \sup_{v \in Y}\{ \langle v,v^*+z^* \rangle_Y -G_K(v)\},$$
and
$$F^*(z^*)= \sup_{v \in Y}\{ \langle  v,z^* \rangle_Y -F(v)\},$$
respectively

Define also $A^*=\{v^* \in Y^* \::\; \Lambda^*v^*-f=\mathbf{0}\}$

$J^*(v^*,z^*)=-G^*_K(v^*+z^*)+F^*(z^*)$

Suppose $(u_0,v_0^*,z_0^*) \in U \times Y^* \times Y^*$ is such that

$$\delta (J^*(v_0^*,z^*_0)+\langle u_0, \Lambda^*v_0^*-f\rangle_U)=\mathbf{0}.$$
Under such hypotheses, we have 
$$\delta J(u_0)=\mathbf{0}$$ and
\begin{eqnarray}J(u_0)&=&\min_{u \in U}\left\{J(u) +\frac{K}{2} \langle \Lambda u -\Lambda u_0,\Lambda u-\Lambda u_0\rangle_Y\right\}\nonumber \\ &=& \sup_{v^* \in A^*}\left\{ J^*(v^*,z^*_0)\right\} \nonumber \\
&=& J^*(v_0^*,z_0^*).\end{eqnarray}

\end{thm}

\begin{proof}

Observe that from the variation of $J^*$ in $u$ we obtain $$\Lambda^* v_0^*-f=\mathbf{0}$$ so that $v_0^* \in A^*.$

Moreover from the variation of $J^*$ in $v^*$ we have
$$\frac{\partial G_K^*(v_0^*+z^*_0)}{\partial v^*}=\Lambda u_0.$$

Also, from the variation of $J^*$ in $z^*$ we have
$$- \frac{\partial G_K^*(v_0^*+z^*_0)}{\partial z^*}+\frac{\partial F^*(z_0^*)}{\partial z^*}=\mathbf{0}.$$
Therefore $$\Lambda u_0=\frac{\partial F^*(z_0^*)}{\partial z^*},$$
so that from the Legendre transform properties $$ z_0^*=\frac{\partial F(\Lambda u_0)}{\partial v}=K \Lambda u_0$$ where $v=\Lambda u.$ Hence,
$$F^*(z_0^*)=\langle \Lambda u_0,z_0^*\rangle -F(\Lambda u_0)=\frac{K}{2}\langle \Lambda u_0,\Lambda u_0 \rangle_Y.$$

Also from the Legendre transform properties we may obtain

$$v_0^*+z_0^*= \frac{\partial G_K(\Lambda u_0)}{\partial v},$$
 so that

\begin{eqnarray}v_0^*&=&\frac{\partial G_K(\Lambda u_0)}{\partial v}-z_0^*
\nonumber \\ &=&\frac{\partial G_K(\Lambda u_0)}{\partial v}-K\Lambda u_0 \nonumber \\ &=& \frac{\partial G(\Lambda u_0)}{\partial v}.
\end{eqnarray}

From this and $$\Lambda^*v_0^*-f=\mathbf{0}$$ we have

$$\Lambda^*\left(\frac{\partial G(\Lambda u_0)}{\partial v}\right)-f=\mathbf{0},$$
that is
$$\delta J(u_0)=\mathbf{0}.$$

Once more through the Legendre transform properties, we get

$$G_K^*(v_0^*+z_0^*)=\langle \Lambda u_0,v_0^*+z_0^* \rangle_Y-G_K(\Lambda u_0),$$
and
$$F^*(z_0^*)=\langle \Lambda u_0,z_0^* \rangle_Y-F(\Lambda u_0),$$
so that \begin{eqnarray}
J^*(v_0^*,z_0^*)&=& -G_K^*(v_0^*+z_0^*)+F^*(z_0^*) \nonumber \\ &=&
-\langle u_0,\Lambda^* v_0^*\rangle_U+G_K(\Lambda u_0)-F(\Lambda u_0) \nonumber \\ &=&
-\langle u_0,f\rangle_U+G(\Lambda u_0) \nonumber \\ &=& J(u_0).
\end{eqnarray}

Moreover, we have
\begin{eqnarray}\label{us33.1}
J^*(v_0^*,z_0^*)&=&\{-G_K^*(v^*_0)+F^*( z^*_0)\}\nonumber \\ &\leq&  -\langle \Lambda u,v_0^*+z_0^*\rangle_Y+G_K(\Lambda u)+F^*(z^*_0) \nonumber \\ &=& G(\Lambda u) +\frac{K}{2}\langle \Lambda u,\Lambda u\rangle_Y-\langle u,f \rangle_Y-\langle \Lambda u,z_0^*\rangle_Y +\frac{K}{2}\langle \Lambda u_0,\Lambda u_0 \rangle_Y \nonumber \\ &=&G(\Lambda u) -\langle u,f \rangle_U+\frac{K}{2}\langle \Lambda u,\Lambda u\rangle_Y-K \langle \Lambda u,\Lambda u_0\rangle_Y +\frac{K}{2} \langle \Lambda u_0,\Lambda u_0\rangle_Y
\nonumber \\ &=& G(\Lambda u)-\langle u,f\rangle_U+\frac{K}{2} \langle \Lambda u-\Lambda u_0,\Lambda u-\Lambda u_0 \rangle_{Y}\nonumber \\   &=& J(u)+\frac{K}{2} \langle \Lambda u-\Lambda u_0,\Lambda u-\Lambda u_0 \rangle_{Y},\; \forall u \in U.
\end{eqnarray}
Summarizing, we have got
\begin{equation}\label{us33.5}
J^*(v_0^*,z_0^*)\leq \inf_{u \in U} \left\{J(u)+\frac{K}{2} \langle \Lambda u-\Lambda u_0,\Lambda u-\Lambda u_0 \rangle_{Y}\right\}\end{equation}
Therefore, from  $$\delta J(u_0)=\mathbf{0},$$
$$J(u_0)=J^*(v_0,z_0^*),$$ from (\ref{us33.5}) and the concavity of $J^*$  in $v^*$,  we have
\begin{eqnarray}J(u_0)&=&\min_{u \in U}\left\{J(u)+\frac{K}{2} \langle \Lambda u-\Lambda u_0,\Lambda u-\Lambda u_0 \rangle_{Y}\right\} \nonumber \\ &=& \sup_{v^* \in A^*} J^*(v^*,z^*_0) \nonumber \\
&=& J^*(v_0^*,z_0^*).\end{eqnarray}
The proof is complete.
\end{proof}


\begin{thebibliography}{}
%
%
\bibitem{1}
R.A. Adams and J.F. Fournier, {Sobolev Spaces}, 2nd edn.
 (Elsevier, New York, 2003).

 \bibitem{100}
J.F. Annet, \textit{Superconductivity, Superfluids and Condensates}, 2nd edn.
 ( Oxford Master Series in
Condensed Matter Physics, Oxford University Press, Reprint, 2010)
\bibitem{85} W.R. Bielski and J.J. Telega, A Contribution to Contact Problems for a Class of Solids and Structures,
Arch. Mech., 37, 4-5, pp. 303-320, Warszawa 1985.

\bibitem{2900} W.R. Bielski, A. Galka, J.J. Telega, The Complementary Energy Principle and Duality for
Geometrically Nonlinear Elastic Shells. I. Simple case of moderate rotations around a tangent to the middle surface.
Bulletin of the Polish Academy of Sciences, Technical Sciences, Vol. 38, No. 7-9, 1988.

\bibitem{55}
D. Bohm, {Quantum Theory}
(Dover Publications INC., New York, 1989).


 \bibitem{120}
F.S. Botelho, {Functional Analysis and Applied Optimization in Banach Spaces},
 (Springer Switzerland, 2014).
\bibitem{700} F.S. Botelho, Functional Analysis, Calculus of Variations and Numerical Methods in Physics and Engineering, Taylor and Francis, Florida, 2020.
\bibitem{15} F. Botelho, {\it Dual Variational Formulations for a
Non-linear Model of Plates}, Journal of Convex Analysis, 17 , No.
1, 131-158 (2010).
\bibitem{360} F. Botelho, {\em Variational Convex Analysis}, Ph.D. thesis, Virginia Tech, Blacksburg, VA -USA, (2009).


\bibitem{[6]} I.Ekeland and R.Temam, {\it Convex Analysis and
Variational Problems.} North Holland (1976).

\bibitem{101}
L.D. Landau and E.M. Lifschits, {Course of Theoretical Physics, Vol. 5- Statistical Physics, part 1}.
(Butterworth-Heinemann, Elsevier, reprint 2008).
\bibitem{12} J.F. Toland, {\it A duality principle for non-convex
optimisation and the calculus of variations}, Arch. Rath. Mech.
Anal., {\bf 71}, No. 1 (1979), 41-61.




\end{thebibliography}
\end{document}